\documentclass[8pt]{article}

\usepackage[T2A]{fontenc}
\usepackage[utf8]{inputenc}

\usepackage{amsmath, amssymb, amsthm, amsfonts, mathtools}
\usepackage{float}

\usepackage[pdftex]{graphicx}

\usepackage[margin = 1.2cm]{geometry}
\usepackage{wrapfig}
\usepackage{pgf,tikz,pgfplots}

\usepackage{hyperref}

\usepackage{mathrsfs}
\usetikzlibrary{arrows}
\usetikzlibrary{patterns}

\mathtoolsset{showonlyrefs}

\newtheorem{theorem}{Theorem}

\newtheorem{prop}{Proposition}
\newtheorem{corr}{Corollary}
\newtheorem{lemma}{Lemma}
\newtheorem{definition}{Definition}
\newtheorem{note}{Note}
\newtheorem{problem}{Problem}

\def\Int{\operatorname{Int}}

\def\max{\operatorname{max}}

\def\diam{\operatorname{diam}}

\def\Sl{\operatorname{Slice}}

\DeclareMathOperator*{\Argmin}{Arg\,min}

\def\eps{\varepsilon}

\title{On the chromatic numbers of 3-dimensional slices}
\author{D. D. Cherkashin, A. J. Kanel-Belov\thanks{The work of A. J. Kanel-Belov is supported by grant RNF 22-11-00177.}, G. A. Strukov, V. A. Voronov}

\begin{document}

\maketitle

\begin{abstract}
We prove that for an arbitrary $\varepsilon > 0$ holds 
\[
\chi (\mathbb{R}^3 \times [0,\varepsilon]^6) \geq 10,
\]
where $\chi(M)$ stands for the chromatic number of an (infinite) graph with the vertex set $M$ and the edge set consists of pairs of points at the distance 1 apart.
\end{abstract}

\section{Introduction}

We study colorings of a set $\Sl(n, k, \varepsilon)=\mathbb{R}^n \times [0,\varepsilon]^{k}$ 
in a finite number of colors with the forbidden distance 1 between monochromatic points; further such sets are called \textit{slices}. 
Slightly abusing the notation we say that $n$ is the \textit{dimension} of a slice.

Define graph $G(n,k,\varepsilon)$, which vertices are the point of $\Sl(n, k, \varepsilon)$ and edges connect points at the Euclidean distance 1 apart. Put
\[
\chi [\Sl(n, k, \varepsilon)] := \chi[G(n,k,\varepsilon)],
\]
where $\chi(H)$ is the chromatic number of a graph $H$.
Obviously for every positive $\varepsilon$ one has 
\[
\chi(\mathbb{R}^n) \leq \chi[\Sl(n,k,\varepsilon)] \leq  \chi(\mathbb{R}^{n+k}).
\]
Since $\chi(\mathbb{R}^n) = (3+o(1))^n$ (see~\cite{LR}), the chromatic number of a slice is finite.
So by the de Bruijn--Erd{\H o}s theorem it is achieved on a finite subgraph.

\subsection{Nelson--Hadwiger problem and its planar generalizations}

In this notation the classical Nelson--Hadwiger problem is to determine $\chi[\Sl(2,0,0)]$, but as usual we write $\chi(\mathbb{R}^2)$ for this quantity. The best known bounds up to the date are
\[
5 \leq \chi(\mathbb{R}^2) \leq 7.
\]
The upper bound is a classical coloring of a regular hexagon tiling due to Isbell.
The lower bound were obtained by de~Grey~\cite{de2018chromatic} in 2018, breaking a 70 year-old record (another constructions are contained in~\cite{heule2018computing,exoo2018chromatic,voronov2021constructing,parts2020chromatic,parts2020graph}).

The study of slice colorings started at~\cite{kanel2018chromatic} with the following main result.
\begin{theorem}
\label{mainold}
For every positive $\varepsilon$ holds
\[6 \leq \chi[\Sl(2,2,\varepsilon)].\]   
\end{theorem}

Theorem~\ref{mainold} is a strengthening of the result $\chi_{\eps}(\mathbb{R}^2) \geq 6$ (Currie--Eggleton~\cite{currie2015chromatic}), where $\chi_{\eps}$ stands for the minimal number of colors, for which there is a coloring of plane without a pair of monochromatic points at the distance in the range $[1, 1 +\eps]$. 
Exoo~\cite{exoo2005varepsilon} conjectured that for every $\eps > 0$ holds
$\chi_{\eps}(\mathbb{R}^2) \geq 7$. Recently Voronov~\cite{voronov2023chromatic} proved this conjecture.

On the other hand Isbell' coloring implies inequality 
\[
\chi_{\eps}(\mathbb{R}^2) \leq 7
\]
for $\eps < 0.13\dots$.
As a corollary, for every $k$ there is $\varepsilon_k > 0$ such that for every positive $\varepsilon < \varepsilon_k$ holds
\[
\chi[\Sl(2,k,\varepsilon)] \leq 7.
\]

\paragraph{Structure of the paper.} The results on real three-dimensional and rational two-dimensional slices are 
stated in Subsections~\ref{resR} and~\ref{resQ}, respectively. 
Section~\ref{tech} contains some auxiliary lemmas.
The proofs are contained in Section~\ref{pr1},~\ref{pr2} and~\ref{pr3}. 
Section~\ref{diss} is devoted to open questions.

\subsection{The chromatic numbers of real 3-dimensional slices}
\label{resR}

First recall the best bounds on $\chi(\mathbb{R}^3)$.
The best known lower bound $\chi (\mathbb{R}^3) \geq 6$ is due to Nechushtan~\cite{Nech}.
The upper bound $\chi (\mathbb{R}^3) \leq 15$ is obtained independently by Coulson~\cite{Coul} and 
by Radoi{\v{c}}i{\'c} and T{\'o}th~\cite{radoivcic2003note}.

The main result of this paper is the following theorem.

\begin{theorem} 
There is $\varepsilon_0 > 0$, such that for every positive $\varepsilon < \varepsilon_0$ holds
\[
10 \leq \chi[\Sl(3,6,\varepsilon)] \leq 15.
\]
\label{mainres}
\end{theorem}

The upper bound immediately follows from the coloring of $\chi (\mathbb{R}^3)$ from~\cite{Coul,radoivcic2003note}, similarly to 2-dimensional case.
The lower bound requires somewhat more complicated arguments than in two dimensions. %def dim layer 

The following theorem is a quantitative strengthening of Theorem 10 from~\cite{kanel2018chromatic} and is of an independent interest.

\begin{theorem}
Let $T \subset \mathbb{R}^n$ be a regular simplex with the edge length $a = \sqrt{2n(n+1)}$. 
Then every proper coloring of $\mathbb{R}^n$ in a finite number of colors contains a point from $T$ belonging to the closures of at least $n+1$ colors. 
\label{point_c4}
\end{theorem}

\begin{corr}
For every positive $\varepsilon'$ holds
  \[
  \chi_{\varepsilon'}(\mathbb{R}^3) \geq 10.
  \]
\end{corr}

Indeed, the orthogonal projection of a unit distance graph from the proof of Theorem~\ref{mainres} for a fixed $\varepsilon$
 has distances between adjacent vertices in the range $[\sqrt{1-6\varepsilon^2}, 1]$. 

\subsection{The chromatic numbers of 2-dimensional rational slices}
\label{resQ}

Denote by $[0, \eps]_\mathbb{Q}$ the set of rational numbers from $[0, \eps]$.
In paper~\cite{kanel2018chromatic} it is shown that $\chi(\mathbb{Q} \times [0, \eps]_\mathbb{Q}^3)  = 3$.
Benda and Perles~\cite{BP} show that $\chi(\mathbb{Q}^4) = 4$. 
Thus the chromatic number of $\mathbb{Q}^2 \times [0, \eps]^2_\mathbb{Q}$ is at most 4.

\begin{prop}
\label{chiQ}
For every $\varepsilon > 0$ holds
\[
\chi (\mathbb{Q}^2 \times [0, \eps]^2_\mathbb{Q}) = 4.
\]
\end{prop}

\section{Notation and auxiliary lemmas}
\label{tech}

Here and after we focus on the following $\Sl(3,6,\varepsilon)\subset \mathbb{R}^9$. By $S_r^n(x)$ we denote a $n$-dimensional sphere of the radius $r$ and centered at $x$.

\begin{definition}
An \textit{attached sphere} of a simplex with vertices $\{v_i\}_{1\leq i \leq k}$, $3 \leq k \leq 4$ is a set of points at the distance 1 from each $v_i$:
\[
S(v_1, \dots , v_k; 1) := \bigcap \limits_i S(v_i; 1) \subset \mathbb{R}^{9}.
\]
Note that if the radius $r$ of a circumscribed $(k-2)$-dimensional sphere $v_1, \dots, v_k$ is smaller than 1, then
$S(v_1, \dots , v_k; 1)$ is a $(9-k)$-dimensional sphere with the radius $\sqrt{1-r^2}$.
\end{definition}

\begin{definition}
A $t$-equator of a sphere $S$ is a subsphere of the dimension $t$ which radius is equal to the radius of $S$. 
\end{definition}

%\begin{definition}
%Equator is called \textit{admissible} if it  называется допустимым, если он лежит во внутренности слойки.
%\end{definition}

As usual, $\overline{T}$ stands for the closure of a set $T$.

\begin{definition}
Let a metric $X$ be colored in a finite number of colors; denote these colors by $C_1,\dots,C_m$.
A \textit{chromaticity} of a point $x \in X$ is the number of sets $\overline{C_i}$, $1 \leq i \leq m$ containing $x$.
\end{definition}

\begin{lemma}[Knaster--Kuratowski--Mazurkiewicz]
Suppose that $(n-1)$-dimensional simplex is covered by closed sets $X_1$, \dots, $X_n$ in such a way that every face $I \subset [n]$ is contained in the union of $X_i$ over $i \in I$. Then all sets $X_i$ have a common point.
\label{KKM}
\end{lemma}

%We need the following technical statement.

%\begin{lemma}
%    Пусть $x_1, \dots, x_{m} \in S^{m-1}(v; 1)$, $B = B(v; 1)$ --- открытый шар единичного радиуса, и функция 
%    \[
%    F: \, \overline{B}^m \to 2^B, \quad \quad \quad F(x'_1, \dots, x'_m):= B \cap  S(x'_1, \dots , x'_m; 1).
%    \]
%    Пусть симплекс $x_1, \dots , x_m$ невырожден. Тогда $F$ непрерывна в некоторой окрестности $(x_1, \dots, x_m)$ (в смысле топологии, индуцированной метрикой Хаусдорфа на $2^B$).
%\end{lemma}

The following lemma is a spherical analogue of the planar lemma from~\cite{kanel2018chromatic}. 
The proof is also analogous; we provide it in the interest of completeness.
\begin{lemma}
Let $S^2_r$ be a sphere of radius $r > \sqrt{\frac{1}{2}}$, $\varepsilon$ be a positive number, 
and $Q \subset S^2_r$ be a $\varepsilon$-neighbourhood of a curve $\xi \subset S^2_r$, such that
\[
\diam \xi > \frac{\sqrt{4r^2-1}}{r}.
\]
Then $\chi(Q) \geq 3$.
\label{bydlo}
\end{lemma}

\begin{proof}
Without loss of generality $\varepsilon < 1$. Denote by $G(Q)$ the corresponding graph; we are going to find an odd cycle in $G(Q)$.

Consider a point $u \in \xi$. Since $\diam \xi > \frac{\sqrt{4r^2-1}}{r} = \diam S(u;1)$,
the intersection of $S(u;1)$ and $\xi$ is non-empty. Let $v \in S(u;1) \cap \xi$; consider such points $v_1$, $v_2$, $v_3$, $v_4 \in S^2_r$ that
$\lVert u-v_1 \rVert = 1$; $\lVert v_i-v_{i+1} \rVert = 1$; $i = 1,2,3$.
If the angles at the vertices of polygonal chain $v u v_1 v_2 v_3 v_4$ are at most $\frac{\varepsilon}{2}$, then $\lVert v - v_1 \rVert < \frac{\varepsilon}{2}$, $\lVert u - v_2 \rVert < \frac{\varepsilon}{2}$, $\lVert v - v_3 \rVert < \varepsilon$, $\lVert u - v_4 \rVert < \varepsilon$, and hence $v_i \in Q$, $i = 1,2,3,4$.

\begin{figure}[ht]
  \centering
  \includegraphics[width=10cm]{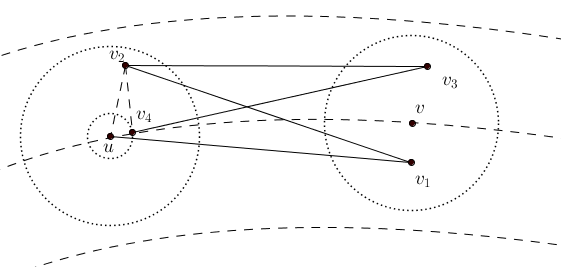}\\
  \caption{A path of length four between $u$ and $v_4$.}
\end{figure}

\noindent Note that
\[
l_1 = \lVert u - v_2 \rVert \in \left [0; 2 \sin \frac{\varepsilon}{4} \right],
\]
\[
l_2 = \lVert v_2 - v_4 \rVert \in \left [0; 2 \sin \frac{\varepsilon}{4} \right]
\]
can be chosen arbitrarily, and the oriented angle between vectors $\overrightarrow{v_2 u}$ and $\overrightarrow{v_2 v_4}$ can be independently chosen from $[-\frac{\varepsilon}{4}; \frac{\varepsilon}{4}]$. 
Fix the line containing vector $\overrightarrow{v_2 u}$; one may choose it orthogonal to $u v$. 
Then a set of all possible $v_4$ contains a rhombus centered at $u$ with the side length $2 \sin \frac{\varepsilon}{4}$ and the angle $\frac{\varepsilon}{2}$. 
Then $G(Q)$ contains a path of length 4 between $u$ and an arbitrary point from a $\gamma$-neighbourhood of $u$, where $\gamma =  \sin  \frac{\varepsilon}{2} \sin \frac{\varepsilon}{4}$.

Thus one may move from $u$ to $v$ along $\xi$ by steps of size at most $\gamma$. 
Every such step corresponds to a path of length 4 in $G(Q)$; since $v$ is adjacent to $u$ we find a desired odd cycle in $G(Q)$.
\end{proof}

\begin{lemma}
    Suppose that a sphere $S_r^2 \subset \mathbb{R}^3$, $r>\sqrt \frac{1}{2}$ has a proper coloring a finite number of colors.
    Then it has a point with the chromaticity at least 3.
\label{c3}
\end{lemma}

\begin{proof}
Note that for $r > \sqrt{\frac 12}$
\[
\frac{\sqrt{4r^2-1}}{r} < 2r.
\]

By compactness of the sphere it is sufficient to show that  there is a spherical ball with an arbitrarily small radius containing points of at least 3 colors.
Suppose the contrary: there is a proper coloring of the sphere and $\varepsilon>0$ such that every spherical ball with the radius $\varepsilon$ 
is colored in at most 2 colors.
Consider a partition of the sphere onto cells such that every cell contains a ball with the radius $\delta = \varepsilon_0/100$ and is contained in a ball
of the radius  $\varepsilon_0/10$.
Then every cell contains points of at most two colors, moreover all the adjacent cells are colored in the same two colors.

Consider an arbitrary cell with two colors (say, colors 1 and 2). Let $A_0$ be the region which is maximal by inclusion, that contains cells with one- or two-colored cells of colors 1 and 2.
The diameter of $A_0$ is smaller than $\frac{\sqrt{4r^2-1}}{r} < 2r$ otherwise, by looking at any path between diametrally opposed points, we have a contradiction with Lemma~\ref{bydlo}.
Let us consider the outer boundary $p$ of the reigon $A_0$. Every cell adjacent to $p$ is adjacent to some cell not in $A_0$, hence it is monochromatic; moreover, colors of all cells from $A_0$ along $p$ are the same, otherwise there would be a ball that contains cells of two different colors and cell not from $A_0$, which contradicts our assumption.
So we may assume that all cells from $A_0$ along $p$ are of color 1. Same argument shows that cells not from $A_0$ along $p$ cannot contain two different colors that are not 1 or 2, and cannot contain the color 2. Therefore all cells adjacent to $p$ are colored in colors 1 or 3 (maybe both). Consider the region $A_1$, that contains cells along $p$ of colors 1 and 3, and is maximal by inclusion. We can apply to $A_1$ the smae argument, and by induction we obtain the sequence of 2-colored regions $A_i$. Note that (spherical) diameter of $A_i$ is increasing by at least $\delta$ each step, so eventually we obtain the contradiction to Lemma~\ref{bydlo}.
\end{proof}

The proof of the main result require technical statement on stability of circumscribed circle of a triangle with vertices of a form $(0,0,0,b_1,\ldots,b_6)$ with respect to a shifts by vectors from the main subspace $\mathbb{R}^3$, i.e. vectors of type $(p,q,r, 0, \dots, 0)$. Such shifts will be called \textit{orthogonal}. The next lemma will be applied for the case of $S^5$, but we prove it in the general case. 

\begin{lemma}
    Suppose that several points are chosen on $S^k$ so that minimal distance between two chosen points is $\Omega(m^{-2})$. Then there is the triangle $\mathcal T$ which vertices are amongst the chosen points satisfying the following condition. Every orthogonal shift of it vertices by $O(m^{-3})$ causes change of radius $R$ of circumscribed circle of $\mathcal T$ by $O\bigl(\frac R{m^2}\bigr)$ and shift of its center by $O(\frac Rm)$.
    \label{kostul}
\end{lemma}

\begin{proof}
Let us find a triangle $\mathcal T_0$ from selected points with heights $\Omega(m^{-2})$. Assume the contrary, id est that there is no such triangle. Let us consider the maximum distance between these points; say, it is achieved between points $A$ and $B$.
Then all other points should lie in the $o(m^{-2})$-neighborhood of the great circle $AB$ (any great circle $AB$, if the points $A$ and $B$ were diametrically opposite): indeed, otherwise the height from point $C$ to $AB$ is equal to $\Omega(m^{-2})$ and it is the smallest of the heights of triangle $ABC$, since points $A$ and $B$ were chosen at the maximum distance and $ABC$ is suitable for the role of $\mathcal{T}_0$.

Let us consider the projections of the selected points onto the great circle $AB$ (they are uniquely determined). Since the pairwise distances between the selected points are equal to $\Omega (m^{-2})$, and the points lie in the $o(m^{-2})$-neighborhood of the great circle $AB$, the projections are separated from each other by at least by $\Omega (m^{-2})$.
One of the arcs $AB$ contains the projection of at least $m_1 \geq m/2$ points.
Let us number the points according to the projection on this arc $AB$; let $C$ be the point with the number $[m_1/2]$. 
Then $AC$ and $BC$ are equal to $\Omega (1/m)$.
Let $O$ be the circumcenter of triangle $ABC$.
Let us denote the lengths of the sides $AB$, $BC$, $AC$ by $c$, $a$, $b$, respectively; let the lengths of the heights be equal to $h_a$, $h_b$, $h_c$.

It is clear that $\angle ACB$ is the largest of the angles of triangle $ABC$ and
\[
\angle ACB \leq \angle OCA + \angle OCB = \arccos \frac{b}{2R} + \arccos \frac{a}{2R} \leq
2\arccos \frac{\Omega (1/m)}{2R},
\]
since the triangles $ACO$ and $BCO$ are isosceles.
Then
\[
2\arccos \frac{\Omega (1/m)}{2R} = \pi - \frac{\Omega (1/m)}{R}.
\]
Therefore, $\sin \angle ACB = \sin (\pi - \angle ACB) = \Omega (1/m)$.
Since $AB$ is the longest side, the height from point $C$ is the smallest. Then, since the sine of at least one of the angles $A$ and $B$
is also equal to $\Omega (1/m)$, the height from point $C$ is equal to $\Omega (1/m^2)$.

The triangle $\mathcal T_0 = ABC$ has been found; let us show that it is suitable as $\mathcal T$. Let us keep the notation for the parameters of the triangle $\mathcal{T}_0$ introduced above. Let the shifted points be $A'$, $B'$ and $C'$. Let us denote by $\Delta q$ the change in the value of $q$ during the transition from $ABC$ to $A'B'C'$.

Let us show that an orthogonal shift of the ends of the segment $y_1y_2$ by $O(m^{-3})$ changes (increases) the length of the segment $l$ by $O(m^{-6}l^{-1})$.
Let us denote the shifted points $z_1$, $z_2$, respectively. Due to orthogonality, $(y_i - y_j, z_j - y_j)=0$.
The square of the new length is
\[
(z_1 - z_2, z_1 - z_2) = \|(z_1-y_1)+(y_1-y_2)+(y_2-z_2)\|^2 =
\]
\[
  = (z_1 - y_1, z_1 - y_1) + (y_1 - y_2, y_1 - y_2) + (z_2 - y_2, z_2 - y_2) - 2 (z_1 - y_1, z_2 - y_2);
\]
That is, the difference in the squares of the lengths is estimated as
\[
  (z_1 - z_2, z_1 - z_2) - (y_1 - y_2, y_1 - y_2) = O(m^{-6}).
\]
It remains to apply the difference of squares formula.

It turns out that $\Delta a = O(m^{-6}a^{-1})$, similarly for other sides.
Let $H$ be the base of the height $CH$, since $AB$ is the greatest, $H$ belongs to the segment $AB$.
Note that the length of the height $h_c$ cannot decrease, and on the other hand, the distance from the shifted vertex $C$ to the point that is projected into $H$ changes by no more than $O(m^{-6}h_c^ {-1})$, and the length of the new height $h'_c$ does not exceed this distance.
Let $S$ be the area of triangle $ABC$, then
\[
\Delta S = O(c \Delta h_C + \Delta c \cdot h_C) = O(h_c m^{-6}c^{-1}) + O(cm^{-6}h_c^{-1} ),
\]
hence,
\[
\frac{\Delta S}{S} = O(m^{-6}c^{-2}) + O(m^{-6}h_c^{-2}) = O(m^{-2 }).
\]

Using the well-known formula
\[
R = \frac{abc}{4S},
\]
we get
\[
\Delta R = O \left (\max \left (\frac{\Delta a \cdot bc}{S}, \frac{\Delta b \cdot ac}{S}, \frac{\Delta c \cdot ab }{S}, \frac{abc \Delta S}{S^2}\right ) \right ) =
O \left (\max \left (\frac{\Delta a }{a}R, \frac{\Delta b }{b}R, \frac{\Delta c }{c}R, \frac{\Delta S}{S}R\right ) \right) = O\left(\frac{R}{m^2}\right),
\]
which is what was required.

Let us limit the shift of the center of the circumscribed circle when changing along one coordinate.
We showed above that the heights and sides of a triangle change slightly when the vertices are orthogonally shifted by $O(m^{-3})$, which means it will be possible to repeat the following estimate several times.

Let us consider three-dimensional Cartesian coordinates in which $C$ is the center, the plane $ABC$ is generated by the first two coordinates, and the latter corresponds to the infinitesimal shift. Then the normal to the plane $ABC$ has the form
\[
\bar v_1 = \bar {AC} \times \bar{BC} = (0,0,2S).
\]
Then the normal to the plane $A'B'C'$ is equal to
\[
\bar v_2 = \bar {A'C'} \times \bar{B'C'} = (A_y B'_z - A'_z B_y, A_x B'_z - A'_z B_x ,2S).
\]
Without loss of generality, $a \geq b$ and then $|A_x|, |B_x|, |A_y|, |B_y| \leq a$. Therefore $|A_y B'_z - A'_z B_y|, |A_x B'_z - A'_z B_x| = O(am^{-3})$.
Recall that $S = 0.5 a h_a = \Omega (am^{-2})$, which implies $|A_y B'_z - A'_z B_y|, |A_x B'_z - A'_z B_x| = O(Sm^{-1})$.
Let us estimate the angle $\phi$ between the planes $ACB$ and $A'B'C'$:
\[
\cos \phi = \frac{(v_1,v_2)}{\sqrt{(v_1,v_1) \cdot (v_2,v_2)}} = \frac{4S^2}{2S \cdot \sqrt{4S^2 + O(S^2m^{-2})}} = 1 - O(m^{-2}), \quad \quad \phi = O(m^{-1}).
\]
Consequently, the change in the center does not exceed $O(R \sin \phi) = O(R/m)$.

\end{proof}

\section{Proof of Theorem~\ref{point_c4}}
\label{pr1}

Suppose the contrary. Let $C_i$~--- be a set of points $\mathbb{R}^n$ of color $i$, $1\leq i \leq m$.
Define
\[
{C^*_i:=\overline{\Int \, \overline{C_i}}}\qquad \mbox{(the closure of the interior of the closure)}.
\]
Split every $C^*_i$ into connected components (with respect to the standard topology): 
\[
C^*_i = \bigcup_{\alpha\in A_i} D_{\alpha}.
\]
Put also $\{ D_\alpha \} = \bigcup\limits_{i=1}^m \bigcup\limits_{\alpha \in A_i} D_\alpha$.

\paragraph{(i)} {\it Sets $C^*_i$ cover $\mathbb{R}^n$.}
Suppose the contrary, i.e. $\exists v: \forall i \ \ v\notin C^*_i$. Then there is an open ball $B(v; \eta)$ such that
\[
B(v; \eta)\cap C^*_i =\emptyset; \ \ B(v; \eta)\subset\bigcup C_i.
\]
Consider an arbitrary ball
\[
B^1 \subset B(v; \eta) \setminus \overline{C}_1.
\]
Then $B^1$ cannot be a subset of $\overline{C}_i$, otherwise the intersection of the interior of $\overline{C}_i$ and $B(v; \eta)$ is nonempty. Define a sequence of balls \[
B^{k+1} \subset B^k \setminus \overline{C}_k.
\]
Note that points of $B^{m+1}$ do not belong to any $\overline{C}_i$, which is a contradiction.

\paragraph{(ii)} {\it Suppose that a point $v \in T$ belongs to exactly $k$ sets $C^*_i$. Assume that $k \leq n$ (otherwise the chromaticity of $v$ is at least $n+1$). Then for every $\mu_0>0$ there is $\mu<\mu_0$ such that the sphere $S(v; 1-\mu)$ does not intersect at least one of those $k$ sets.}

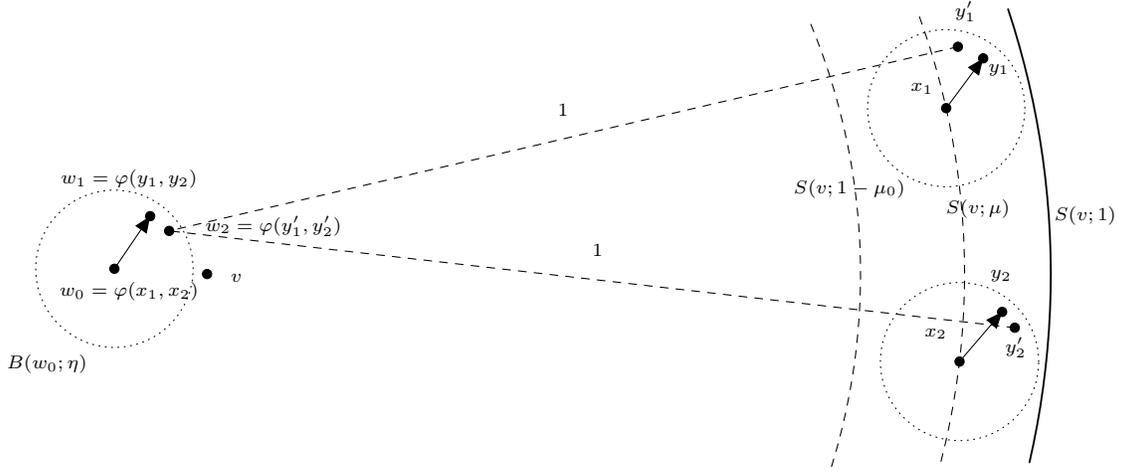
\begin{figure}[ht]
  \centering
  \begin{tikzpicture}[line cap=round,line join=round,>=triangle 45,x=1cm,y=1cm, scale=0.7]
%\clip(-15.32,-11.91) rectangle (9,7.35);
\draw [shift={(-8.32,-1.17)},line width=0.4pt,dash pattern=on 3pt off 3pt]  plot[domain=-0.29932821702574053:0.4022887552879237,variable=\t]({1*12.411897518107374*cos(\t r)+0*12.411897518107374*sin(\t r)},{0*12.411897518107374*cos(\t r)+1*12.411897518107374*sin(\t r)});
\draw [shift={(-8.32,-1.17)},line width=0.7pt]  plot[domain=-0.22530212325946675:0.3197384846772764,variable=\t]({1*16.02500546021748*cos(\t r)+0*16.02500546021748*sin(\t r)},{0*16.02500546021748*cos(\t r)+1*16.02500546021748*sin(\t r)});
\draw [shift={(-8.32,-1.17)},line width=0.4pt,dash pattern=on 3pt off 3pt]  plot[domain=-0.25003593626230725:0.3537648719359465,variable=\t]({1*14.38739726288254*cos(\t r)+0*14.38739726288254*sin(\t r)},{0*14.38739726288254*cos(\t r)+1*14.38739726288254*sin(\t r)});
\draw [line width=0.4pt,dotted] (5.97089385654596,-2.833596339541816) circle (1.5028885105088576cm);
\draw [line width=0.4pt,dotted] (5.7176668792211505,1.9829523605689863) circle (1.4934274487630022cm);
\draw [line width=0.4pt,dotted] (-10.08,-1.07) circle (1.494523335381552cm);
\draw [->,line width=0.4pt] (-10.08,-1.07) -- (-9.4,-0.07);
\draw [->,line width=0.4pt] (5.97089385654596,-2.833596339541816) -- (6.78,-1.89);
\draw [->,line width=0.4pt] (5.7176668792211505,1.9829523605689863) -- (6.42,2.93);
\draw [line width=0.4pt,dash pattern=on 3pt off 3pt] (7.02,-2.19)-- (-9.04,-0.35);
\draw [line width=0.4pt,dash pattern=on 3pt off 3pt] (5.94,3.15)-- (-9.04,-0.35);
\begin{scriptsize}
\draw [fill=black] (-8.32,-1.17) circle (2.5pt);
\draw[color=black] (-7.76,-1.24) node {$v$};
\draw[color=black] (3.88,0.46) node {$S(v; 1-\mu_0)$};
\draw[color=black] (8.36,-0.06) node {$S(v; 1)$};
\draw[color=black] (6.32,0.08) node {$S(v; \mu)$};
\draw [fill=black] (5.97089385654596,-2.833596339541816) circle (2.5pt);
\draw[color=black] (5.52,-2.3) node {$x_2$};
\draw [fill=black] (5.7176668792211505,1.9829523605689863) circle (2.5pt);
\draw[color=black] (5.26,2.36) node {$x_1$};
\draw [fill=black] (-10.08,-1.07) circle (2.5pt);
\draw[color=black] (-9.8,-1.5) node {$w_0 = \varphi(x_1,x_2)$};
\draw [fill=black] (6.42,2.93) circle (2.5pt);
\draw[color=black] (6.72,2.72) node {$y_1$};
\draw [fill=black] (6.78,-1.89) circle (2.5pt);
\draw[color=black] (6.76,-1.2) node {$y_2$};
\draw [fill=black] (-9.4,-0.07) circle (2.5pt);
\draw[color=black] (-9.8,0.64) node {$w_1 = \varphi(y_1,y_2)$};
\draw [fill=black] (7.02,-2.19) circle (2.5pt);
\draw[color=black] (7.04,-2.56) node {$y'_2$};
\draw [fill=black] (5.94,3.15) circle (2.5pt);
\draw[color=black] (6.08,3.82) node {$y'_1$};
\draw [fill=black] (-9.04,-0.35) circle (2.5pt);
\draw[color=black] (-7.06, - 0.25) node {$w_2 = \varphi(y'_1,y'_2)$};
\draw[color=black] (-0.9,-0.72) node {1};
\draw[color=black] (-1.56,1.96) node {1};
\draw[color=black] (-11.36,-2.88) node {$B(w_0; \eta)$};
\end{scriptsize}
\end{tikzpicture}
  \caption{Illustration to item~(ii).}
\end{figure}

\noindent 

We can assume without loss of generality that $v \in C^*_i$, ${1\leq i\leq k}$. Suppose the contrary, i.e. there is such a $\mu_0>0$ that for every $\mu \in (0, \mu_0)$ holds
\[
    S(v; 1-\mu) \cap C^*_i \neq \emptyset, \quad 1 \leq i \leq k.
\]
By the definition of $C^*_i$ any neighbourhood of an arbitrary point $x \in C^*_i$ contains a point from $\operatorname{Int} \overline{C}_i$. Hence, the set 
\[
    \mathcal{M}_0:= \{\mu \in (0, \mu_0) \; |\; \exists S(v; 1-\mu) \cap \operatorname{Int} \overline{C}_i = \emptyset \}
\]
is closed and nowhere dense. 

Fix some $\mu \in (0, \mu_0) \setminus \mathcal{M}_0$. One may choose points $x_1, ... , x_k$ in such a way that
\[
x_i\in S(v; 1-\mu)  \cap \Int \overline{C}_i, \quad \ 1\leq i \leq k;
\]
and $\{v, x_1, ... , x_k\}$ are in a general position (i.e. all the simplices are non-degenerate). 
Consider any $\eta > 0$ such that $B(x_i; \eta) \subset C^*_i$, $1 \leq i \leq k$. Put
\[
z \in B(0; \eta); \quad y_i = x_i + z.
\]

Define 
\[
w_0 = \phi (x_1, \dots, x_k) := \Argmin_{u \in U} \|u - v\|, \quad U =  \bigcap\limits_{1 \leq i \leq k} S(y_i; 1),
\]
\[
w_1 = \phi (y_1, \dots, y_k).
\]
By the construction the color of $w_1$ differs from the colors of $y_1, \dots , y_k$.

In a small neighbourhood of $\{y_i\}$ function $w(\cdot)$ is properly defined and continuous. Choose points
\[
y'_i \in C_i, \quad 1 \leq i \leq k,
\]
for which exists $w_2 = \phi(y'_1, \dots, y'_k)$. Then
\[
w_2 \in \bigcup_{j=k+1}^{m} {C}_j.
\]
At the same time the quantity
\[
\delta(y'_1, ... ,y'_k) = \max_{1 \leq i \leq k} \lVert y'_i - y_i \rVert
\]
can be chosen arbitrarily small and hence 
\[
w_1 \in \bigcup_{j=k+1}^{m} \overline{C_j}.
\]
Since $z \in B(0; \eta)$ was chosen arbitrarily
\[
B(w; \eta) \subset \bigcup_{j=k+1}^{m} \overline{C_j}.
\]

Hence an arbitrary neighbourhood of $w_0$ has an inner point of at least one set $\overline{C_j}$, $k+1 \leq j \leq m$, so $w_0 \in C_j^*$ for some $j$.
Note that $w_0 = \phi(y_1, \dots, y_k) \to v$ with $\mu \to 0$, thus $v$ belongs to at least one of sets $C_j^*$, $k+1 \leq j \leq m$, which contradicts to the initial assumption.

\paragraph{(iii)} {\it There is a cover of $T$ by sets from $\{D_\alpha\}$, such that every set from the cover is contained in a closed unit ball.}
By (ii) every point is covered by at least one set that satisfies the condition.
Axiom of choice finishes the proof of the item. For every color $i$ denote by $\{D^{(i)}_\beta\}$ the chosen sets. 

\paragraph{(iv)} {\it There is a finite cover of $T$ by closed sets $D'_{ik}$, $1 \leq i \leq m$, $1 \leq k \leq K_i$,
such that every set from the cover is the union of some sets from $\{D_\alpha\}$ and also is contained in a closed unit ball.}

For every $i$, $1 \leq i \leq m$ consider a sequence $v^i_1, v^i_2, \dots \in \bigcup D^{(i)}_\beta$ such that 
\[
\gamma(v_j^i) = i;
\]
\[
    v^i_{s+1} \in \bigcup D^{(i)}_\beta \setminus \bigcup_{1 \leq j \leq s}  B(v^i_j; 1). 
\]
Let the sequence be maximal (with respect to inclusion). 
The pairwise distances $v^i_j$, $j=1,2,\dots$ are at least 1, so the sequence is finite because $T$ is bounded. 
Now let us define 
\[
    D'_{ik} = B(v^i_k; 1) \cap \left( \bigcup D^{(i)}_\beta\ \right) \setminus \bigcup_{1\leq j \leq k-1} D'_{ij}.  
\]
Every set $D'_{ik}$ is separated from other connected components of $C^*_i$ by a neighbourhood of a sphere,
without points from $C^*_i$ (see Fig.~\ref{cover}). Thus these sets are closed.

\begin{figure}[ht]
  \centering

    \begin{tikzpicture}[line cap=round,line join=round,>=triangle 45,x=1cm,y=1cm, scale=0.6]
%\clip(-6.36,-14.39) rectangle (27.84,4.87);
\draw [line width=0.6pt,dash pattern=on 2pt off 2pt] (3.42,-2.87) circle (4.524820438426259cm);
%\draw [line width=0.8pt,dash pattern=on 2pt off 2pt] (3.42,-2.87) circle (3.9597979746446663cm);
\draw [rotate around={-46.54815769897798:(3.46,-0.13)},draw=none,fill=black,pattern=north east lines,pattern color=black] (3.46,-0.13) ellipse (0.7333497720806829cm and 0.5136164796915984cm);
\draw [rotate around={82.5685920288265:(4.48,-5.43)},draw=none,fill=black,pattern=north east lines,pattern color=black] (4.48,-5.43) ellipse (0.9490713489389879cm and 0.8279712708644338cm);
\draw [rotate around={-7.125016348901787:(2.12,-4.66)},draw=none,fill=black,pattern=north east lines,pattern color=black] (2.12,-4.66) ellipse (1.3756482360612015cm and 1.1687206977624272cm);
\draw [rotate around={45:(1.26,-1.65)},draw=none,fill=black,pattern=north east lines,pattern color=black] (1.26,-1.65) ellipse (0.8906971032054911cm and 0.7671644736682308cm);
\draw [rotate around={67.8336541779168:(2.81,-1.48)},draw=none,fill=black,pattern=north east lines,pattern color=black] (2.81,-1.48) ellipse (0.6737363660894117cm and 0.6073884185522189cm);
\draw [rotate around={15.945395900922858:(3.37,-2.79)},draw=none,fill=black,pattern=north east lines,pattern color=black] (3.37,-2.79) ellipse (0.6573345461081204cm and 0.5473469699442631cm);
%\draw [line width=0.8pt] (9.46,-2.85) circle (4.471822894525229cm);
%\draw [line width=0.8pt,dash pattern=on 1pt off 1pt] (9.46,-2.85) circle (3.961262425035736cm);
\draw [rotate around={73.30075576600636:(9.42,-2.99)},draw=none,fill=black,pattern=north east lines,pattern color=black] (9.42,-2.99) ellipse (0.8932719463366976cm and 0.6368161195448689cm);
\draw [rotate around={-17.65012421993008:(11.1,-1.75)},draw=none,fill=black,pattern=north east lines,pattern color=black] (11.1,-1.75) ellipse (1.0414373990583705cm and 0.9334837203494613cm);
\draw [rotate around={-70.70995378081128:(8.78,-0.41)},draw=none,fill=black,pattern=north east lines,pattern color=black] (8.78,-0.41) ellipse (0.773865360110162cm and 0.6475087610051541cm);
\draw [rotate around={-58.13402230639642:(10.95,-3.82)},draw=none,fill=black,pattern=north east lines,pattern color=black] (10.95,-3.82) ellipse (0.7658953953945954cm and 0.6299172617786087cm);
\draw [rotate around={-30.784146526326413:(5.81,2.37)},draw=none,fill=black,pattern=north east lines,pattern color=black] (5.81,2.37) ellipse (1.0685428172782818cm and 0.9178691368365138cm);
\draw [rotate around={-39.98688624496419:(11.15,2.77)},draw=none,fill=black,pattern=north east lines,pattern color=black] (11.15,2.77) ellipse (0.839144180642144cm and 0.735161856944148cm);
\draw [rotate around={-39.98688624496407:(9.17,-5.33)},draw=none,fill=black,pattern=north east lines,pattern color=black] (9.17,-5.33) ellipse (0.8391441806421522cm and 0.7351618569441547cm);
\draw [rotate around={-46.54815769897799:(6.44,-3.17)},draw=none,fill=black,pattern=north east lines,pattern color=black] (6.44,-3.17) ellipse (0.7333497720806627cm and 0.5136164796915842cm);
%\draw [shift={(9.46,-2.85)},line width=0.6pt, dash pattern=on 2pt off 2pt]  plot[domain=-2.2915651105645582:2.298187602916966,variable=\t]({1*3.9612624250357347*cos(\t r)+0*3.9612624250357347*sin(\t r)},{0*3.9612624250357347*cos(\t r)+1*3.9612624250357347*sin(\t r)});
\draw [shift={(9.46,-2.85)},line width=0.6pt, dash pattern=on 2pt off 2pt]  plot[domain=-2.297012134077718:2.3036346264301257,variable=\t]({1*4.4718228945252285*cos(\t r)+0*4.4718228945252285*sin(\t r)},{0*4.4718228945252285*cos(\t r)+1*4.4718228945252285*sin(\t r)});

\begin{scriptsize}
\draw [fill=black] (3.42,-2.87) circle (1.5pt);
\draw[color=black] (2.38,-2.72) node {$v^1_1$};
\draw[color=black] (-0.6,1.52) node {$S(v^1_1; 1-\mu_1^1)$};
%\draw[color=black] (2.06,-0.14) node {$S(v^1_1; 1-\mu)$};
\draw[color=black] (0.68,-3.56) node {$D'_{11}$};
\draw[color=black] (10.48,-5.44) node {$D'_{12}$};
\draw [fill=black] (9.46,-2.85) circle (1.5pt);
\draw[color=black] (8.94,-2.1) node {$v^1_2$};
\end{scriptsize}
\end{tikzpicture}
  \caption{Illustration to item (iv). The construction of sets $D'_{ik}$}
  \label{cover}
\end{figure}
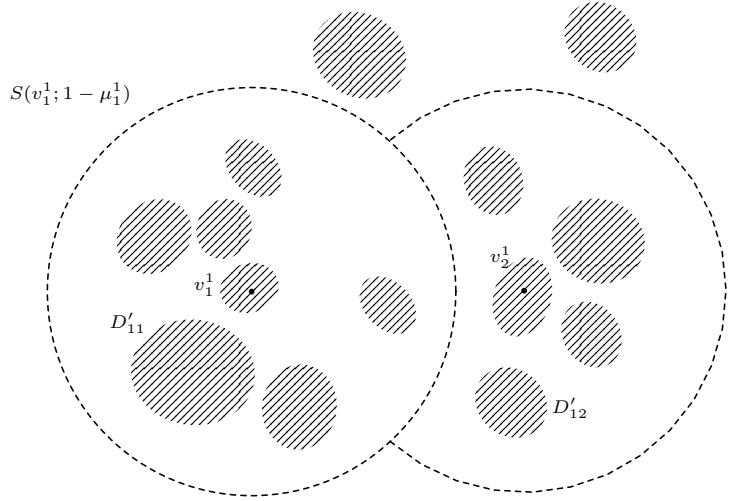

Come back to the main construction and note that every set $D'_{ik}$ cannot intersect every face of simplex $T$, 
because it is contained in an open unit ball while the inner radius of $T$ is equal to 1.
Split the cover $\mathcal{D}' = \bigcup_{i} \{D'_{ik}\}$ into $n+1$ subfamilies, in the way that every set subfamily consists of sets that do not intersect a face of $T$. Clearly there is a bijection between subfamilies and the vertices of $T$. Let $X_i$, $i = 1,\dots, n+1$ be the unions of sets over corresponding subfamilies. 
By Lemma~\ref{KKM} sets $X_i$ have a common point $x_*$, and thus an arbitrary neighbourhood of $x_*$ intersects with at least $n+1$ sets from $\{D_\alpha\}$. 
They belong to at least $n+1$ different $\{C^*_i\}$, because $\{D_\alpha\}$ are connected components. Hence, $x^*$ has the chromaticity at least $n+1$.

\section{Proof of the main result}
\label{pr2}

\paragraph{Outline of the proof.} 
Suppose the contrary to the statement.
First, we find points $v_1, v_2, v_3, v_4$ of different colors, such that the intersection $I$ of attached sphere $S(v_1,v_2,v_3,v_4;1)$ and the slice contains 2-equator $S^2$ of the sphere and we also require the radius of the sphere to be close to 1.

Then $I$ is close (in the sense of Hausdorff distance) to $S^2_{1-\eta} \times [0, \eps]^3$, where $\eta$ is small enough.
Then one can follow the arguments from~\cite{kanel2018chromatic}, that were applied in the case of 2-dimensional slices.
Note that the sets of colors of $I$ and $\{v_1, v_2, v_3, v_4\}$ are disjoint.

Let us find points $v_5, v_6, v_7 \in I$, such that an equator of the corresponding attached sphere belongs to the slice. 
The attached sphere of $v_1, \dots, v_7$ has an equator belonging to the slice, so the intersection of $v_1, \dots, v_7$ contains a spherical neighbourhood of a circle. It requires 3 additional colors in addition to the colors of points $v_1, \dots, v_7$.

\begin{figure}[ht]
  \centering
  \begin{tikzpicture}[line cap=round,line join=round,>=triangle 45,x=1cm,y=1cm]
%\clip(-6.36,-9.43) rectangle (10.06,4.87);
\draw [line width=1pt] (1.44,-1.21) circle (4.66244571013969cm);
\draw [rotate around={-0.252959855370859:(1.45,-1.15)},line width=0.6pt,dotted] (1.45,-1.15) ellipse (4.650585080653886cm and 1.051970338175196cm);
\draw [rotate around={-74.94352857261201:(3.98,-0.53)}, color = lightgray, line width=3.8pt,dash pattern=on 3pt off 3pt] (3.98,-0.53) ellipse (3.732580860995107cm and 1.1787111112851052cm);
\draw [line width=0.4pt,dash pattern=on 3pt off 3pt] (1.18,-0.89)-- (5.56,0.35);
\draw [line width=0.4pt,dash pattern=on 3pt off 3pt] (1.44,-1.21)-- (5.56,0.35);
\draw [line width=0.4pt,dash pattern=on 3pt off 3pt] (1.78,-1.29)-- (5.56,0.35);
\draw [line width=0.4pt,dash pattern=on 3pt off 3pt] (1.78,-1.29)-- (5.74,-0.07);
\draw [line width=0.4pt,dash pattern=on 3pt off 3pt] (1.12,-1.39)-- (5.56,0.35);
\draw [line width=0.4pt,dash pattern=on 3pt off 3pt] (1.12,-1.39)-- (5.663944521931169,0.7640042238170786);
\draw [line width=0.4pt,dash pattern=on 3pt off 3pt] (1.18,-0.89)-- (5.74,-0.07);
\draw [line width=0.4pt,dash pattern=on 3pt off 3pt] (1.78,-1.29)-- (5.663944521931169,0.7640042238170786);
\draw [line width=0.4pt,dash pattern=on 3pt off 3pt] (5.56,0.35)-- (2.5342779411968177,2.255579945629563);
\draw [line width=0.4pt,dash pattern=on 3pt off 3pt] (5.663944521931169,0.7640042238170786)-- (2.5342779411968177,2.255579945629563);
\draw [line width=0.4pt,dash pattern=on 3pt off 3pt] (5.74,-0.07)-- (2.5342779411968177,2.255579945629563);
\draw [line width=0.4pt,dash pattern=on 3pt off 3pt] (1.18,-0.89)-- (2.5342779411968177,2.255579945629563);
\draw [line width=0.4pt,dash pattern=on 3pt off 3pt] (1.44,-1.21)-- (2.5342779411968177,2.255579945629563);
\draw [line width=0.4pt,dash pattern=on 3pt off 3pt] (1.78,-1.29)-- (2.5342779411968177,2.255579945629563);
\draw [line width=0.4pt,dash pattern=on 3pt off 3pt] (1.12,-1.39)-- (2.5342779411968177,2.255579945629563);
\begin{scriptsize}
\draw [fill=black] (1.44,-1.21) circle (1.5pt);
\draw[color=black] (2.26,-1.28) node {$v_3$};
\draw [fill=black] (1.18,-0.89) circle (1.5pt);
\draw[color=black] (0.64,-0.42) node {$v_4$};
\draw [fill=black] (1.12,-1.39) circle (1.5pt);
\draw[color=black] (0.74,-1.54) node {$v_1$};
\draw [fill=black] (1.78,-1.29) circle (1.5pt);
\draw[color=black] (1.54,-1.4) node {$v_2$};
\draw [fill=black] (5.663944521931169,0.7640042238170786) circle (1.5pt);
\draw[color=black] (5.88,1.18) node {$v_7$};
\draw [fill=black] (5.74,-0.07) circle (1.5pt);
\draw[color=black] (6.16,0.1) node {$v_5$};
\draw [fill=black] (5.56,0.35) circle (1.5pt);
\draw[color=black] (6.02,0.62) node {$v_6$};
\draw [fill=black] (2.5342779411968177,2.255579945629563) circle (1.5pt);
\draw [fill=black] (2.534277941196824,2.255579945629577) circle (2pt);
\end{scriptsize}
\end{tikzpicture}
  \caption{Construction of a rainbow 10-point set.}
\end{figure}
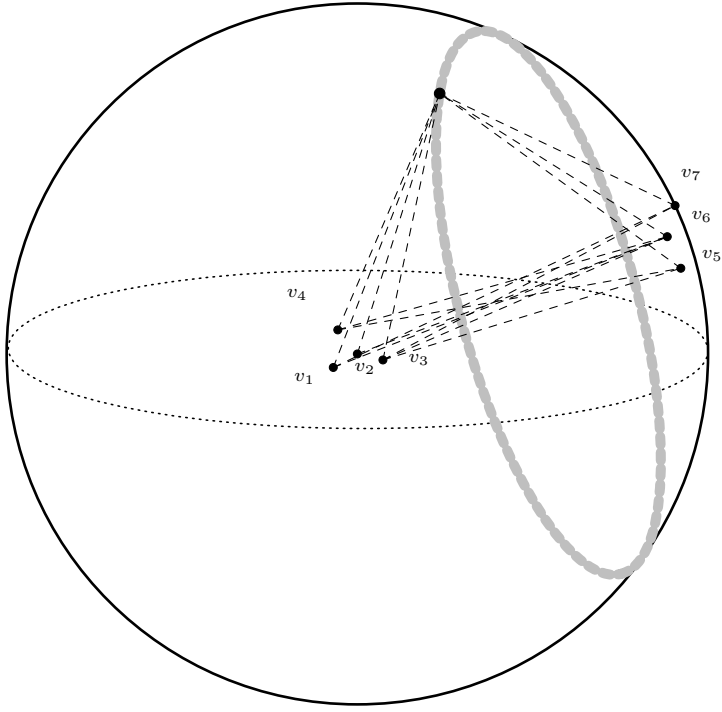

\paragraph{Step 1.} Finding of points $v_1$, $v_2$, $v_3$, $v_4$, 
which attached sphere has the radius closed to 1 and the great circle belonging to the slice. 

This requires the 3-dimensional subspace $U$ spanned by $v_1$, $v_2$, $v_3$, $v_4$, 
to be ``almost orthogonal'' to the main subspace $\mathbb{R}^3$, 
and the circumradius of the simplex $v_1v_2v_3v_4$ in $U$ to be small enough.

Consider the standard Cartesian coordinate system in slice $\mathbb{R}^3 \times [0,\eps]^6$: 
\[
v = (x_1,x_2,x_3, y_1, \dots , y_6), \quad  x_i \in \mathbb{R}, \quad y_i \in [0, \eps].
\]
For a given point $v=(x_1,x_2,x_3, y_1, \dots , y_6)$ define projections 
\[
p_R(v) = (x_1,x_2,x_3, 0, \dots , 0) \quad \mbox{ and } \quad p_\eps(v) = (0,0,0, y_1, \dots , y_6).
\]  
Consider sphere $S := S^5_{\eps_1}$ of the radius $\eps_1 < \eps/2$ centered at $(0,0,0, \eps/2, \eps/2, \dots, \eps/2)$; note that  $S \subset (0, 0, 0) \times [0,\eps]^6$. Let $T \subset \mathbb{R}^3$ be an arbitrary regular tetrahedron with the edge length $2 \sqrt{6}$ and the center at the origin and $u$ be an arbitrary point of sphere $S$. By Lemma~\ref{point_c4} every set $T \times \{ u\} \subset T \times S$ has a point with chromaticity at least 4.

Fix the parameters $\delta, h > 0$, which values will be chosen later.

Consider a set of points $U = \{u_1, \dots, u_m\} \subset S$ such that $\|u_i-u_j\| \geq \delta$, $i \neq j$ and $m$ is maximal. Obviously $m = \Omega(\delta^{-5})$. 
Match every point $u_i \in U$ with an arbitrary point $u^*_i \in T \times \{u_i\}$ with chromaticity at least $4$. 

Consider how $T$ is cut by a cubic mesh with edge length $h$:
\[
T_{i,j,k} := \bigsqcup\limits_{i,j,k} T \cap Z_{i,j,k}; \quad \quad
Z_{i,j,k} := [i h, (i+1)h ) \times [j h , (j+1)h) \times [k h , (k+1)h), 
\]
where $i,j,k$ are integers. Since $T \subset \mathbb{R}^3$ is bounded, one has
\[
\# \{(i,j,k): T \cap Z_{i,j,k} \neq \emptyset\}  = O (h^{-3}).
\]

Consider points $w_i = p_R(u^*_i) \in T$. Put $h = \delta^{3/2}$. There is a cell $T_{a,b,c}$ such that it contains at least 
\[
m = \frac{\Omega(\delta^{-5})}{O(\delta^{-9/2})} = \Omega \left (\delta^{-\frac{1}{2}} \right)
\]
points from $\{ w_i \}$. Note that $h = O(m^{-3})$, $\delta = O(m^{-2})$ and
\[
\diam T_{a,b,c} \leq \sqrt{3}h = \sqrt{3}\delta^{3/2} = O(m^{-3}).
\]
Now apply Lemma~\ref{kostul} for these $m$ points. It gives a triangle $\mathcal{T} = w_1w_2w_3$ such that its arbitrary small orthogonal shift, in particular the triangle $u_1^*u_2^*u_3^*$ has circumradius at most $(1/4 + O(m^{-2}))\varepsilon$ and its circumcircle $\omega$ belongs to the slice. Let us construct a (five-dimensional) sphere $S^*$ on $\omega$ as the diameter and choose $v_4$ as the most distant point from the plane $u_1^*u_2^*u_3^*$ on the sphere $S^*$. Then the simplex $u_1^*u_2^*u_3^*v_4$ is a non-degenerate simplex whose circumscribed sphere belongs to the interior of the slice.

It remains to choose in arbitrarily small neighborhoods of the points $u_1^*,u_2^*$ and $u_3^*$ the points $v_1,v_2$ and $v_3$, respectively, in such a way that the points $v_1,v_2,v_3$ and $v_4$ have pairwise different colors.

\paragraph{Step 2.} Finding in sphere $S(v_1,v_2,v_3,v_4; 1)$  points $v_5,v_6,v_7$ of different colors such that attached (2-dimensional) sphere $S(v_1, \dots, v_7;1)$ has a 2-equator belonging to the slice. Note that  $S(v_1, \dots, v_7,1)$ is the intersection of $S(v_1,v_2,v_3,v_4; 1)$ and $S(v_5,v_6,v_7; 1)$. 
A proper choice of $\varepsilon_1$, $h$ makes radii of the spheres and the distance between its centers close to $1$. 
Then the radius of  $S(v_1, \dots, v_7,1)$ tends to $\frac{\sqrt{3}}{2}>\frac{1}{2}$.

Suppose the intersection of an attached sphere with the slice 
\[
M := S(v_1,v_2,v_3,v_4;1) \bigcap \mathbb{R}^3 \times [0, \eps]^6 =  S^5_{1-\eta} \bigcap \mathbb{R}^3 \times [0, \eps]^6
\]
is properly colored and the equator $M_E = S^2_{1-\eta}$ belongs to the slice.

Let $H \subset \mathbb{R}^9$ be the 6-dimensional subspace containing $S^5_{1-\eta}$. Consider a coordinates in $H$ such that $M_E$ belongs to the subspace spanned by the first 3 coordinates. 
For every point $u \in M_E$, $u = (u_1,u_2,u_3,0,0,0)$ consider a sphere
\[
    S^2(u; \nu) = \{\sqrt{1-\nu^2} u + \xi \; | \quad  \xi=(0,0,0,\xi_4,\xi_5,\xi_6); \; \| \xi \| = \nu\}. 
\]
Note that $S^2(u; \nu)$ is a subset of $M$ when $\nu$ is small enough.

For every $u$ consider the following regular pentagon belonging to $S^2(u; \nu)$:
\[
    w_{u,k} = \left (u_1,u_2,u_3, \cos \frac{2\pi k}{5} \nu, \sin \frac{2\pi k}{5} \nu, 0 \right), \quad \quad k = 1, \dots, 5.  
\]

Let $u$ be a point. If one can find among $w_{u,1}, \dots, w_{u,5}$ points of three different colors, then they can be taken as $v_5$, $v_6$, $v_7$. 
Otherwise vertices of every pentagon are colored in at most 2 different colors, i.e. there is a color with at least three representatives. Call this color \emph{dominating} at $u$.

Consider an auxiliary coloring $\pi$ in which every point of $M_E$ has its dominating color. 
Let us show that $\pi$ is proper. Indeed if the distance between $p,q \in M_E$ is equal to 1, then $\|w_{p,k} - w_{q,k}\| = 1$ for every $k$, so by the pigeonhole principle dominating colors at $p$ and $q$ are different.

By Lemma~\ref{c3} sphere $M_E$ has a point $u^*$ with chromaticity at least 3, i.e. an arbitrary neighbourhood of $u^*$ has three points of different dominating colors. Then one may choose from corresponding pentagons 3 points of different colors in a way that chosen points lie in three small neighbourhoods of points $w_{u^*,1}, \dots, w_{u^*,5}$. 
Every triangle with vertices in these points is non-degenerate, and has sides of length at least $\nu$.

\paragraph{Step 3.} Recall that every point from $v_1$, $v_2$, $v_3$, $v_4$ and every point from $v_5$, $v_6$, $v_7$ lie at the 
distance 1 apart. Moreover, $v_1$, $v_2$, $v_3$, $v_4$ have pairwise different colors; the same holds for $v_5$, $v_6$, $v_7$.
Moreover, by Lemma~\ref{bydlo} (applied to equator that lies in the slice) the intersection of attached sphere $S(v_1, \dots, v_7; 1)$ and the slice has the chromatic number at least 3. Hence we show that a proper coloring of the slice requires at least 4+3+3=10 colors, as desired.

\section{Proof of Proposition~\ref{chiQ}}
\label{pr3}

Consider the following 4 points in $\mathbb{Q}^2 \times [0,\varepsilon]_\mathbb{Q}^2$:
\begin{gather}
    A = (0,\ 0,\ 0,\ 0),
    \\
    B = (q,\ \tfrac12,\ \alpha,\ \beta),
    \qquad
    C = (q,\ -\tfrac12,\ \alpha,\ \beta),
    \\
    D = (2q,\ 0,\ 0,\ 0).
\end{gather}
So we have
\begin{equation}\label{dist}
    |AB|^2 = |AC|^2 = |BD|^2 = |CD|^2 = q^2 + \frac14 + \alpha^2 + \beta^2.
\end{equation}
Our goal is to choose numbers $q \in \mathbb{Q}$ and $\alpha, \beta \in [0,\eps]_\mathbb{Q}$ such that expression~\eqref{dist} is equal to 1. Let $q = a/2b$, where $a$ and $b$ are some integers.
Then we need 
\begin{equation}
    \alpha^2 + \beta^2 = \frac34 - \frac {a^2}{(2b)^2} = \frac{3b^2 - a^2}{4b^2}.
\end{equation}

It is enough for $(a,b)$ to satisfy 
\begin{equation}\label{eq:main}
    3b^2 - a^2 = 2,
\end{equation}
so if $b$ is large enough, we can put $\alpha = \beta = \frac{1}{2b}$.

Let us construct a series of solutions to~\eqref{eq:main} as follows.
Given the solution $(a_n,b_n)$, we build next pair as
\begin{equation}\label{eq:recur}
    (a_{n+1},\ b_{n+1}) = (7a_n + 12b_n,\ 4a_n + 7b_n).
\end{equation}
One can check that $(a_{n+1},b_{n+1})$ is a solution to~\eqref{eq:main} by straightforward computation and use of assumption that so is $(a_n,b_n)$.
Now by taking $(a_0,b_0) = (1,1)$ we get an infinite sequence of solutions with $b_n$ strictly increasing without limit. So for any given $\eps$ there is some $n_\eps$ such that for $n>n_\eps$
\begin{equation}
    \frac{3b_n^2 - a_n^2}{4b_n^2} = \frac 1{2b_n^2} < 2\eps^2,
\end{equation}
which implies $1/2b < \eps$

Now we are going to find such integers $x$ and $y$ that 
\begin{equation}
    x\cdot \frac {a_n}{b_n} + y \cdot \frac {a_{n+1}}{b_{n+1}} = 1
\end{equation}
or
\begin{equation}
    x\cdot a_n b_{n+1} + 
    y\cdot a_{n+1} b_n = 
    b_n b_{n+1}.
\end{equation}
So existence of such $x$ and $y$ is equivalent to
\begin{equation}\label{eq:gcddiv}
    \gcd(
        a_n b_{n+1},\
        a_{n+1} b_n
    ) \mid b_n b_{n+1}.
\end{equation}
It is sufficient to show that $\gcd(\ldots) = 1$:
\begin{gather}
    \gcd(
        a_n b_{n+1},\
        a_{n+1} b_n
    ) \mid
    (a_n b_{n+1} - a_{n+1} b_n),\\
    a_n b_{n+1} - a_{n+1} b_n =
    a_n (4a_n + 7b_n) - b_n (7a_n + 12b_n) =
    4a_n^2 - 12b_n^2 = -4(3b_n^2 - a_n^2) = -8.
\end{gather}
And from~\eqref{eq:recur} it is clear that
\begin{gather}
    a_{n+1} \equiv a_n \equiv \ldots \equiv a_0 = 1 \pmod 2,\\
    b_{n+1} \equiv b_n \equiv \ldots \equiv b_0 = 1 \pmod 2.
\end{gather}
So  $\gcd(\dots )=1$ as required.

Finally, let $\chi(\mathbb{Q}^2\times [0,\eps]_\mathbb{Q}^2) = 3$. Then points $A$ and $D$ have the same color. Hence, each point at the distance $k\cdot  a_n/b_n + l\cdot a_{n+1}/b_{n+1}$ (where $1/2b_n^2 < 2\eps^2$ and $k,l$ are integers) from 0 has the same color. Taking $k = x$ and $l = y$, one can obtain that $(1,0,0,0)$ has the same color. A contradiction.

\begin{note}
    \def\Z{\mathbb Z}
    Recursion formula~\eqref{eq:recur} was obtained the following way. Consider a ring $\Z[\sqrt 3]$.
    It has the norm 
    \[
        N(a+b\sqrt 3) = (a+b\sqrt 3)(a-b \sqrt 3)= a^2-3b^2.
    \]
    Then~\eqref{eq:main} transforms to an equation $N(\alpha)=-2$.
    Norm is multiplicative:  $N(\alpha\beta)=N(\alpha)N(\beta)$ for any $\alpha,\beta \in \Z[\sqrt 3]$. So if $N(\alpha)=-2$ and $N(\zeta)=1$, then $N(\alpha\zeta) = -2$. For~\eqref{eq:recur} one can take $\zeta = 7+4\sqrt 3$.
\end{note}

\section{Further questions}
\label{diss}

%We are not able to solve the following problem with a small enough function $\delta(\varepsilon)$, which is an objection to generalization of the result to an arbitrary dimension.

%\begin{problem}
%Consider sets $A_1,A_2, \dots, A_{m}$ in $\mathbb{R}^n$, where $m \geq n+1$.
%It turns out that every choice of $(n+1)$ points from pairwise different sets form a simplex with volume at least $\varepsilon$.
%Show that $A_1,A_2, \dots, A_{m}$ lie in a $\delta(\varepsilon)$-neighbourhood of some hyperplane.
%\end{problem}

\begin{problem}
Let $\mathcal{M}_n$ be a family of compact convex set  $\mathbb{R}^n$ such that a proper coloring of any $\mathbb{R}^n$ have a point of chromaticity at least $n+1$ in every $M \in \mathcal{M}_n$. Evaluate $V^*_n = \inf_{M \in \mathcal{M}_n} \operatorname{Vol} M$ from above. 
\end{problem}
Theorem~\ref{point_c4} gives the bound $V^*_n \leq  \frac{\sqrt{n+1}}{n!\sqrt{2^n}} \cdot \left ( \sqrt{2n(n+1)} \right)^n = \frac{\sqrt{n^n (n+1)^{n+1}}}{n!}$.

\bibliographystyle{plain}
\bibliography{main}

\paragraph{Danila Cherkashin:}  Institute of Mathematics and Informatics, Bulgarian Academy of Sciences
Sofia 1113, 8 Acad. G. Bonchev str.

Email address: jiocb.orlangyr@gmail.com

\paragraph{Alexei Kanel-Belov:}  Department of Discrete Mathematics, Moscow
institute of physics and technology, Dolgoprudny 141700, Russia.

Email address: kanelster@gmail.com

\paragraph{Georgii Strukov}
The Euler International Mathematical Institute, St. Petersburg, Russia

\paragraph{Vsevolod Voronov:} Caucasus Mathematical Center of Adyghe State University, Maikop 385000, Russia.

Email address: v-vor@yandex.ru

\end{document}